\documentclass[12pt]{amsart}
\usepackage{amsfonts,amsmath,amsthm,amscd,amssymb,latexsym}
\usepackage[english]{babel}

\newtheorem{Th}{Theorem}

\newtheorem{Rm}{Remark}

\newtheorem*{Cj}{Conjecture}
\newcommand{\PP}{\mathcal{P}}
\newcommand{\NN}{\mathbb{N}}
\newcommand{\w}{\omega}

\begin{document}

\title{A conjecture on partitions of groups}
\author{Igor Protasov, Sergii Slobodianiuk}
\date{}

\subjclass[2010]{03E05, 20B07, 20F69}

\keywords{partition of groups, covering number.}

\maketitle

\begin{abstract}
We conjecture that every infinite group $G$ can be partitioned into countably many cells $G=\bigcup\limits_{n\in\w}A_n$ such that $cov(A_nA_n^{-1})=|G|$ for each $n\in\w$. Here $cov(A)=\min\{|X|:X\subseteq G, G=XA\}$. We confirm this conjecture for each group of regular cardinality and for some groups (in particular, Abelian) of an arbitrary cardinality.
\end{abstract}

\section*{Introduction}
For any finite partition of an infinite group, at last one cell of the partition has a rich combinatorial structure. The Ramsey Theory of groups gives a plenty of concrete examples (see \cite{b4}, \cite{b7})

On the other hand, the subsets of a group could be classified by their size. For corresponding partition problems see the survey \cite{b8}.

Given a group $G$ and a subset $A$ of $G$, we denote
$$cov(A)=\min\{|X|:X\subseteq G, G=XA\}.$$
The covering number $cov(A)$ evaluates a size of $A$ inside $G$ and, if $A$ is a subgroup, coincides with the index $|G:A|$.

It is easy to partition each infinite group $G=A_1\cup A_2$ so that $cov(A_1)$ and $cov(A_2)$ are infinite. Moreover, if $|G|$ is regular, there is a partition $G=\bigcup\limits_{\alpha<|G|}H_\alpha$ such that $cov(G\setminus H_\alpha)=|G|$ for each $\alpha<|G|$. In particular, there is a partition $G=A_1\cup A_2$ such that $cov(A_1)=cov(A_2)=|G|$. See \cite{b5}, \cite{b9}, \cite{b10} for these statements, their generalizations and applications.

However, for every $n\in\NN$, there is a (minimal) natural number $\Phi(n)$ such that, for every group $G$ and every partition $G=A_1\cup...\cup A_n$, $cov(A_iA_i^{-1})\le\Phi(n)$ for some cell $A_i$ of the partition. It is still an open problem posed in \cite[Problem 13.44]{b6} whether $\Phi(n)=n$. For the history and results behind this problem see the survey \cite{b1}.

In \cite[Question F]{b2}, J. Erde asked whether, given a partition $\PP$ of an infinite group $G$ such that $|\PP|<|G|$, there is $A\in\PP$ such that $cov(AA^{-1})$ is finite. After some simple examples answering this question extremely negatively, we run into the following conjecture.

\begin{Cj} 
Every infinite group $G$ of cardinality $\varkappa$ can be partitioned $G=\bigcup\limits_{n<\w}A_n$ so that $cov(A_nA_n^{-1})=\varkappa$ for each $n\in\w$. 
\end{Cj}

In this note we confirm Conjecture for every group of regular cardinality and for some groups (in particular, Abelian) of an arbitrary cardinality.

\section*{Results}

For a cardinal $\varkappa$, we denote by $cf(\varkappa)$ the cofinality of $\varkappa$.
\begin{Th}
Let $G$ be an infinite group of cardinality $\varkappa$. Then there exists a partition $G=\bigcup\limits_{n\in\w}A_n$ such that $cov(A_nA_n^{-1})\ge cf(\kappa)$ for each $n\in\w$.
\end{Th}
\begin{proof}
If $G$ is countable, the statement is trivial: take any partition of $G$ into finite subsets. For $\varkappa>\aleph_0$, we choose a family $\{G_\alpha:\alpha<\varkappa\}$ of subgroups of $G$ such that
\begin{itemize}
\item[(1)] $G_0=\{e\}$ and $G=\bigcup\limits_{\alpha<\varkappa}G_\alpha$, $e$ is the identity of $G$;
\item[(2)] $G_\alpha\subset G_\beta$ for all $\alpha<\beta<\varkappa$;
\item[(3)] $\bigcup\limits_{\alpha<\beta}G_\alpha=G_\beta$ for each limit ordinal $\beta<\varkappa$;
\item[(4)] $|G_\alpha|<\varkappa$ for each $\alpha<\varkappa$.
\end{itemize}
Following \cite{b11}, for each $\alpha<\varkappa$, we decompose $G_{\alpha+1}\setminus G_\alpha$ into right cosets 
by $G_\alpha$ and choose some system $X_\alpha$ of representatives so $G_{\alpha+1}=G_{\alpha}X_{\alpha}$. 
Take an arbitrary element $g\in G\setminus\{e\}$ and choose the smallest subgroup $G_\alpha$ with $g\in G_\alpha$. 
By $(3)$, $\alpha=\alpha_1+1$ for some ordinal $\alpha_1<\varkappa$. 
Hence $g\in G_{\alpha_1+1}\setminus G_{\alpha_1}$ and there exists $g_1\in G_{\alpha_1}$, $x_{\alpha_1}\in X_{\alpha_1}$ such that $g=g_1x_{\alpha_1}$. If $g_1\neq e$, we choose the ordinal $\alpha_2$ and elements $g_2\in G_{\alpha_2+1}\setminus G_{\alpha_2}$ and $x_{\alpha_2}\in X_{\alpha_2}$ such that $g_1=g_2x_{\alpha_2}$. Since the set of ordinals $\{\alpha:\alpha<\varkappa\}$ is well-ordered, after finite number $s(g)$ of steps, we get the representation 
$$g=x_{\alpha_{s(g)}}x_{\alpha_{s(g)-1}}...x_{\alpha_2}x_{\alpha_1}\text{, } x_{\alpha_i}\in X_{\alpha_i}.$$
We note that this representation is unique and put
$$\gamma_1(g)=\alpha_1,\text{ }\gamma_2(g)=\alpha_2,...,\text{ }\gamma_{s(g)}(g)=\alpha_{s(g)},\text{ }\max(g)=\gamma_1(g).$$

Each ordinal $\alpha<\varkappa$ can be written uniquely as $\alpha=\beta+n$ where $\beta$ is a limit ordinal and $n\in\w$. 
We put $f(\alpha)=n$ and denote by $Seq(\w)$ the set of all finite sequences of elements of $\w$. 
Then we define a mapping $\chi:G\setminus\{e\}\to Seq(\w)$ by 
$$\chi(g)=f(\gamma_{s(g)}(g))f(\gamma_{s(g)-1}(g))...f(\gamma_2(g))f(\gamma_1(g)),$$
and, for each $s\in Seq(\w)$, put $H_s=\chi^{-1}(s)$. Since the set $Seq(\w)$ is countable, it suffices to prove that $cov(H_sH_s^{-1})\ge cf\varkappa$ for each $s\in Seq(\w)$.

We take an arbitrary $s\in Seq(\w)$ and an arbitrary $K\subseteq G$ such that $|K|< cf\varkappa$. 
Then we choose $\gamma<\varkappa$ such that $\gamma>\max g$ for each $g\in K$ and $f(\gamma)\notin\{s_1,...,s_n\}$. 
We pick $h\in X_\gamma$ and show that $KH_s\cap hH_s=\varnothing$.

If $g\in KH_s$ and $\gamma_i(g)\ge\gamma$ then $f(\gamma_i(g))\in\{s_1,...,s_n\}$. 
To prove this we take an arbitrary $g\in KH_s$ and fix $k\in K$, $x\in H_s$ such that $g=kx$.
Let $i\in\w$ be the length of the representation of $x$ after $G_\gamma$, which means that there exist 
$x_1\in X_{\gamma_1(x)},...,x_i\in X_{\gamma_i(x)}$ such that $x=yx_i...x_1$, for some $y\in G_\gamma$.
As $x\in H_s$ it follows that $f(x_j)\in\{s_1,...,s_n\}$, for $1\le j \le i$ and since $k\in K\subseteq G_\gamma$ 
we have that $g=kx=zx_{\alpha_i}...x_{\alpha_1}$, where $z\in G_\gamma$. So for any $i\le s(g)$ for which 
$\gamma_i(g)\ge\gamma$ we get that $f(\gamma_i(g))\in\{s_1,...,s_n\}$.

Now if $g'\in hH_s$ then the representations of $g'$ and $h^{-1}g$ after $G_{\gamma+1}$ are equal with the same
length $i$ due to the previous argument as $h\in G_{\gamma+1}$. But they are different after $G_{\gamma}$ as 
$f(\gamma_{i+1}(h^{-1}g'))\in\{s_1,...,s_n\}$, and then since, $h\notin G_\gamma$, 
$f(\gamma_{i+1}(g'))=f(\gamma)\notin\{s_1,...,s_n\}$. 

Hence $KH_s\cap hH_s=\varnothing$, so $h\notin KH_sH_s^{-1}$ and $cov(H_sH_s^{-1})$ can not be less then $cf\varkappa$.
\end{proof}
\begin{Th}\label{t2}
Let $\lambda,\varkappa$ be infinite cardinals, $\lambda<\varkappa$ and let $\{H_\alpha:\alpha<\varkappa\}$ be a family of groups such that $|H_\alpha|\le\lambda$ for each $\alpha<\varkappa$. Let $G$ be a subgroup of the direct product $H=\otimes_{\alpha<\varkappa}H_\alpha$ such that $|G|=\varkappa$. Then there exists a partition $G=\bigcup\limits_{n<\w}A_n$ such that $cov(A_nA_n^{-1})=\varkappa$ for each $n\in\w$.
\end{Th}
\begin{proof}
For each $g\in G$, $supt(g)$ denotes the number of non-identity coordinates of $g$, $\chi(g)=|supt(g)|$. For each $h\in\w$ we put $$A_n=G\cap\chi^{-1}(n)$$
and show that $cov(A_nA_n^{-1})=\varkappa$.

We take an arbitrary $K\subset G$ such that $|K|<\varkappa$ and denote
$$S=\{\alpha<\varkappa: pr_\alpha g\neq e_{\alpha}\text{ for some }g\in K\},\text{ }T=\varkappa\setminus S,$$
$$G_T=G\cap\otimes_{\alpha\in T}H_\alpha.$$
Since $\lambda<\varkappa$ and $|G|=\varkappa$, we have $|G_T|=\varkappa$. If $g\in KA_nA_n^{-1}\cap G_T$ then $|supt(g)|\le2n$. On the other hand, for every $m\in\w$, there is $h\in G_T$ such that $|supt(h)|>m$. Hence, $G_T\setminus KA_nA_n^{-1}\neq\varnothing$.
\end{proof}
\begin{Rm}\end{Rm} It is well-known \cite[Theorems 23.1 and 24.1]{b3}
that each Abelian group can be embedded into the direct product of countable groups. Applying Theorem~\ref{t2}, we confirm conjecture for Abelian groups. Moreover, if a group $G$ of cardinality $\varkappa$ has an Abelian homomorphic image of cardinality $\varkappa$  (in particular, if $G$ is a free group of rank $\varkappa$) then Conjecture is valid for $G$.

\begin{Rm}\end{Rm}
Every infinite group $G$ can be written as a union $G=\bigcup\limits_{\alpha<cf\varkappa}H_\alpha$ of subsets of cardinality $<\varkappa$ (and so $cov(H_\alpha H_\alpha^{-1})=\varkappa$). Hence, Conjecture holds also if $cf\varkappa=\aleph_0$.

\end{document}